\documentclass[10pt]{amsart}
\usepackage{euler,amsmath,amssymb,amscd}
\usepackage{epsfig}

\newtheorem{proposition}{Proposition}

\newtheorem{lemma}{Lemma}

\theoremstyle{definition}

\newtheorem{remark}{Remark}

\newtheorem{ex}{Exercise}
\newtheorem{definition}{Definition}

\newtheorem*{ack}{Acknowledgement}

\def\til{\widetilde}
\def\mbb{\mathbb}

\def\cal{\mathcal}
\def\mfk{\mathfrak}
\def\ten{\otimes}

\def\tu{\textup}

\def\a{\alpha}

\def\d{\delta}

\def\bP{\mbb P}

\def\bG{\mathbb G}

\def\inj{\hookrightarrow}

\def\spec{\tu{Spec\,}}

\def\om2{\omega^{\ten 2}}

\def\Mg{\overline{M}_g}

\def\inj{\hookrightarrow}

\def\M{\bar{M}}
\def\cO{\mathcal O}

\def\bar{\overline}

\input xy
\xyoption{all}

\input epsf
\usepackage{charter}
\def\Hg{\bar H_g}
\def\Bg{\bar B_{2g+2}}
\def\Mg{\bar M_g}
\def\M{\bar M}
\def\ZZ{\mathbb Z}

\begin{document}

\title[Outline of the log MMP for $\Mg$]{An outline of the log minimal model program for the moduli space of curves}
\date{\today}
\author{Donghoon Hyeon}
%\address{Department of Mathematics\\ Marshall University\\ Huntington, WV 25755}
\address{Department of Mathematics\\ Pohang University of Science and Technology\\ South Korea}
\email{dhyeon@postech.ac.kr}
\begin{abstract} B. Hassett and S. Keel predicted that there is a descending sequence of critical
$\alpha$ values where the log canonical model for the moduli space of stable
curves with respect to $\alpha \delta$ changes, where $\delta$ denotes the divisor of singular curves. 
We derive a conjectural formula
for the critical values in two different ways, by working out the intersection
theory of the moduli space of hyperelliptic curves and by computing the GIT
stability of certain curves with tails and bridges. The results give a rough
outline of how the log minimal model program would proceed, telling us when the
log canonical model changes and which curves are to be discarded and acquired
at the critical steps.
\end{abstract}
\maketitle

\section{Introduction}
 Recently in a series of papers, B. Hassett, Y. Lee and the author completed the first couple of steps of the log minimal model program and showed that the log canonical model
$
\Mg(\a) := Proj \oplus_{m\ge 0} \Gamma(m(K_{\overline {\mathcal M}_g} + \a \delta))
$
changes at certain critical values of $\a$, contracting the locus of elliptic tails and bridges.
Moreover, these new compactifications of $M_g$ were identified with GIT moduli spaces parametrizing ordinary cusps and tacnodes. The whole story so far demonstrates the so-called  Hassett-Keel principle: As we run the log minimal model program for $\bar M_g$ with respect to  $\a\d$, decreasing $\a$ from $1$ to $0$, we produce birational contractions with exceptional loci in the divisors $\delta_i$, and the resulting varieties are moduli spaces parametrizing curves with increasingly worse singularities.  In general, flips are expected in between the divisorial contractions with centers in $\delta_i$, possibly introducing curves with singularities other than rational singularities. And it is the existence of the flips that makes carrying out this program highly nontrivial. Recently, \cite{ASvdW} gave a stack-theoretic construction of the second flip, of which projective moduli space would be proven to exist in a forthcoming work. Also, there is an excellent survey on the Hassett-Keel program by Fedorchuk and Smyth \cite{FS}, which puts the whole program in a broader context of birational geometry of $\Mg$. 

While we are making steady progress in this program, step by step arduously, we  turn our attention in this paper to some work in the literature where Hassett-Keel principle is in clear display. In \cite{Kap, Kap1}, Kapranov constructed birational morphisms from $\overline{M}_{0,n}$ to GIT quotients $(\bP^1)^n/\!\!/_{\bf x}SL_2$ where ${\bf x}$ denotes the linearization $\cO(x_1, \dots, x_n)$. The symmetric linearization case $\cO(1, \dots, 1)$ was worked out in \cite{GvdP}. 
See also \cite{Has1} in which such morphisms are given functorially
 in a more general setting of weighted pointed curves, and \cite{KiemM} where the morphism is explicitly decomposed into a sequence of blow-downs.
 For the symmetric linearization ${\bf x} = (1, \dots, 1)$, the birational morphism descends to give a birational morphism from $\overline{M}_{0,n}/\mathfrak S_n$ to the compact moduli space $\bP^n/\!\!/SL_2$ of semistable binary forms of degree $n$. For the even $n = 2g+2$ case where $\overline{M}_{0,n}/\mathfrak S_n$ is isomorphic to the moduli space of stable hyperelliptic curves, Avritzer and Lange give a detailed functorial description of the birational morphism which they denote by $f_g$ \cite{AL}: On the locus $H_g$ of smooth hyperelliptic curves, $f_g$ is an isomorphism, and it maps a general curve
\[
(R_1\cup R_2, x_1, \dots, x_{2g+2}), \quad x_k \in
\begin{cases} R_1, \quad k\le j \\ R_2, \quad k > j
\end{cases}
\]
in the boundary $\widetilde B_j$ (\S\ref{S:intersection}), $j \le g+1$, 
to the binary form  $(x - x_0)^j\prod_{k=j+1}^{2g+2}(x-x_k)$, where $x_0$ denotes the point of attachment.
Note that for odd $j$, the stable curve associated to a general pointed curve in $\widetilde B_j$ is a {\it genus $\frac{j-1}2$ tail} , i.e. a  curve consisting of a genus $g-\frac{j-1}2$ curve $D$ and a genus $\frac{j-1}2$ curve $T$ meeting in one node $p$. For even $j$, it is a curve of the form $D \cup_{p_1,p_2} B$ consisting of a genus $g-\frac j2-1$ curve $D$ and a genus $\frac j2$ curve $B$ meeting in two nodes $p_1$ and $p_2$. Such a curve will be called a {\it genus $\frac j2$ bridge}. In terms of stable curves, $f_g$ ``replaces" the {\it genus $\frac {j-1}2$ tail} (resp. {\it genus $\frac j2$ bridge})  by a singularity locally analytically isomorphic to $y^2 = x^j$ if $j$ is odd (resp. even), putting forth a picture that is very desirable from the view point of Hassett-Keel principle. Indeed, for $g = 2$, $f_2$ is precisely the divisorial contraction given in \cite{Has} and \cite{HL1}.

Yongnam Lee and the author took the initiative from this and considered
the log minimal model program for the moduli space $\Hg$ of hyperelliptic curves with suitable log canonical divisor $K_{\Hg} + D_\a$, and proved that the boundary divisors $\widetilde B_j$, $j = 3, 4, 5$ are contracted in succession \cite{HL3}. The problem had seemed to get combinatorially too complex very quickly as $j$ gets larger.

In this article, we shall postulate that the log MMP for $\bar H_g$ obeys Hassett-Keel, and deduce a conjectural formula for the critical values $\a_j$ at which the log canonical model $\Hg(\a)$ changes. Furthermore, we also consider the other aspect of Hassett-Keel principle that dictates how $\Hg(\a)$ manifests itself as a moduli space: We  compute the GIT stability of the $m$th Hilbert point of certain curves with singularity locally analytically isomorphic to $y^2 = x^k$, and the results predict that such a curve will go from being GIT unstable to (semi)stable at a critical value. The critical values are then shown to precisely match $\a_j$s via the relation between $\a$ and $m$. We shall make this more precise and state the main results in the remainder of the introduction.

Let $L_\a$ be the $\mathbb Q$-divisor
\[
L_\a  :=  \sum_{s=1}^{\lfloor\frac{g+1}2\rfloor} \left\{
\frac{13}{4g+2} s(g+1-s) + 2(\a-2) \right\} \til B_{2s} +
\sum_{s=1}^{\lfloor\frac{g}2\rfloor} \left\{\frac{13}{4g+2}s(g-s)
+ \frac{\a-2}2\right\} \til B_{2s+1}
\]
\normalsize and define $\bar H_g(\a)$ by
%\begin{equation}\label{E:projnLa}
\[
\bar H_g(\a) := Proj \oplus_{n\ge 0} \Gamma(\bar H_g, n (4g+2) L_\a).
\]
%\end{equation}
Hassett-Keel principle gives us a sequence of critical values $\a_j$ and birational maps $f_{g,j} : \Hg(\a_{j-1}) \dashrightarrow \Hg(\a_j)$ precisely at which $\tilde B_j$ gets contracted:
\begin{center}
\[
\xymatrix{
\overline H_g \ar[r]^-{f_{g,3}}\ar@/_2pc/[rrrr]_{f_g} & \overline H_g(9/11) \ar@{->}[r]^-{f_{g,4}}
& \Hg(7/10) \ar@{-->}[r]^-{f_{g,5}} & \cdots  \ar@{-->}[r]^-{f_{g, g+1}} & \overline H_g(\a_{g+1}) %\ar[r]^-{\simeq} & \overline B_{2g+2}
}
\]
\end{center}

\begin{proposition}\label{P:main1}
\begin{enumerate}
\item (conjectural critical values) Critical values $\a_j$ are given by
\begin{equation*}\label{E:critical-values}\tag{$\diamondsuit$}
\a_j = \left\{\begin{array}{llll}
1, & j = 2\\
\frac{3j^2+ 10 j - 21}{8j^2 - 8j -4}, &  j = 3, 5, 7, \dots\\
\frac{3j+16}{8j+8}, &  j = 4, 6, 8, \dots
\end{array}\right.
\end{equation*}

\item (conjectural discrepancy formula) The pullback by $f_{g, \le j} := f_{g,j}\circ f_{g,j-1} \circ \cdots \circ f_3$ of the cycle theoretic image of $L_\a$ on $\Hg(\a_j)$ is given by
\begin{equation*}\label{E:Lai}
L_\a^{[j]} = \sum_{k=2}^j \binom k2 r_2 \widetilde B_k + \sum_{k=j+1}^{g+1} r_k \widetilde B_k. \tag{$\heartsuit$}
\end{equation*}
where $r_k$ are $\widetilde B_k$-coefficient of $L_\a$.
In particular, $$L_\a^{[g+1]} = r_2 \sum_{k=2}^{g+1} \binom k2 \widetilde B_k.$$

\item There is an isomorphism $\Hg(\a_{g+1}) \simeq \Bg$ via which the map $f_g$  is identified with $f_{g, g+1}\circ f_{g,g} \circ \cdots \circ f_{g,4}\circ f_{g,3}$ away from a codimension $\ge 2$ locus.
\end{enumerate}
\end{proposition}

 The proposition is not only interesting on its own, but also provides useful information on the log minimal model program for $\bar M_g$. For instance, the critical values in this conjecture are expected to be the  critical values for the log MMP for $\bar M_g$ as well. Indeed, $\a_3 = 9/11$, $\a_4 = 7/10$, $\a_5 = 2/3$ and $\a_6 = 17/28$ have appeared as critical values in the aforementioned work. Also, our analysis of log MMP for $\bar H_4$ in \cite{HL3} plays a crucial role in our upcoming paper on the log MMP for $\M_4$ \cite{HL4}.

Hassett-Keel principle predicts that as $\delta_i$ is contracted at, say $\a = \a^\star$, genus $i$ tails $D\cup_p T$
are replaced in $\Mg(\a^\star)$ by suitable curves with $A_{2i}$ singularity $y^2 = x^{2i+1}$.
%Since such a curve must depend only on $(D,p)$, a natural candidate is the curve $C'$ obtained from $D$ by replacing the tail $T$ by an $A_{2i}$ singularity $y^2 = x^{2i+1}$, or a curve that is semistably equivalent to $C'$.
{\it How does $\a$ factor in the GIT construction of the log canonical models}, so that the stability of $D\cup_p T$ changes at the critical value $\a^\star$?
Recall that the GIT stability depends on how the parameter space is linearized. The Hilbert scheme of $\nu$-canonically embedded curves naturally admits a one-parameter family of linearizations, and $\a$ can be regarded as the parameter in the variation of GIT quotients. We shall consider the stability of $D\cup_p T$ with rational cuspidal tail $T$ and show that it is stable with respect to a one-parameter subgroup coming from ${\rm Aut}(T)$ for $\a > \a_{2i+1}$ but becomes unstable for $\a < \a_{2i+1}$, where $\a_j$ are precisely the critical values obtained by intersection theory in Proposition~\ref{P:main1}. The subscript $2i+1$ comes from the relation $\delta_i \cap \Hg = \widetilde B_{2i+1}$. This strongly suggests how the log canonical models can be constructed as GIT quotients. To state the results precisely, recall that $\mu^L(x, \rho)$ denotes the Hilbert-Mumford index with respect to a 1-PS $\rho : \bG_m \to G$ of a point $x$ in a projective variety with $L$-linearized $G$ action. Also, for any projective variety $X$, $[X]_m$ denotes its $m$th Hilbert point.

%In the ensuing propositions, we present several GIT computations that

\begin{proposition}\label{P:tails}
 Let $C = D\cup_p T$ be a bicanonical genus $i$ tail, $i\ge 2$, such that $T$ is a rational curve with an $A_{2i}$ singularity  $y^2 = x^{2i+1}$. Then there exists a one parameter subgroup $\rho$ of $SL_{3g-3}$ coming from ${\rm Aut}(T)$ such that
\[
\mu([C]_m, \rho) =  \frac13(m-1) ((4i^2-8i+2)m - 3i^2)
\]
Using the relation (\ref{E:anm}) between $m$ and $\a$ (\S\ref{S:GIT}), we may write it in terms of $\a$  and $j = 2i+1$:
\[
\begin{array}{cllll}
\mu([C]_m, \rho) & = &
 \frac{17\a-8}{8(7-10\a)^2}\left((8j^2 - 8j - 4)\a - (3j^2+ 10 j - 21)\right) \\
 & = & \frac{(8j^2 - 8j - 4)(17\a-8)}{8(7-10\a)^2}(\a - \a_j).
 \end{array}
 \]
Thus $C$ is Hilbert stable with respect to $\rho$ for  $\a > \a_{2i+1}$, strictly semistable for $\a = \a_{2i+1}$ and unstable for $\a < \a_{2i+1}$.

Retain $C$ and $\rho$, and consider the basin of attraction $B_\rho([C]_m)$.
We obtain:
\begin{enumerate}
\item Let $C' = D \cup_{p'} T'$ be a bicanonical genus $i$ tail, $i\ge 2$, where $T'$ is a hyperelliptic
curve of genus $i$ and $p'$ is a Weierstrass point of $T'$. Then
\[
\mu([C']_m, \rho) = \mu([C]_m, \rho).
\]
 In particular $C'$ is $\a$-Hilbert stable with respect to $\a
> \a_{2i+1}$, strictly semistable for $\a = \a_{2i+1}$ and unstable for $\a < \a_{2i+1}$.  $T'$ (and sometimes $C'$ itself by abusing terminology) is called a
\emph{Weierstrass genus $i$ tail}.

\item Let $C''$ be a bicanonical genus $g$ curve obtained from $C'$ by replacing $T'$ by
an $A_{2i}$ singularity. That is, $C''$ is of genus $g$, has $A_{2i}$ singularity at $p''$ and admits a
partial normalization $\nu : (D, p) \to (C'', p'')$. Then
\[
\mu([C'']_m, \rho) = -\mu([C']_m, \rho)
\]
In particular, $C''$ is $\a$-Hilbert unstable with respect to $\a
> \a_{2i+1}$, strictly semistable for $\a = \a_{2i+1}$ and stable otherwise.
\end{enumerate}

\end{proposition}

We also consider the stability of genus $i$ bridges.  Genus one bridges (elliptic bridges) were introduced in \cite{HH2}, where we showed that they are Hilbert unstable and Chow strictly semistable when they are bicanonically embedded.
Bridges of higher genera also become Hilbert unstable as $\a$ decreases further:

\begin{proposition}\label{P:bridges} Let $C = D\cup_{p_1,p_2} B$ be a bicanonical genus $i$ bridge, $i\ge 2$, such that $B = R_1\cup R_2$ is a union of two rational curves meeting in one point forming an $A_{2i+1}$ singularity $y^2 = x^{2i+2}$. Then there exists a one parameters subgroup $\rho$ of $SL_{3g-3}$ coming from ${\rm Aut}(R)$ such that
\[
\mu([C]_m, \rho) =  \frac16(m-1)(4i(i-1)m-3i(i+1)).
\]
Using the relation (\ref{E:anm}), \S\ref{S:GIT}, we may write it in terms of $\a$  and $j = 2i+2$:
\[
\begin{array}{cllll}
\mu([C]_m, \rho) & = &
 \frac{(j-2)(17\a-8)}{16(7-10\a)^2}\left(8(j+1)\a - (3j+16)\right) \\
 & = & \frac{(j^2-j-2)(17\a-8)}{2(7-10\a)^2}(\a - \a_j).
 \end{array}
 \]
Thus $C$ is Hilbert stable with respect to $\rho$ for  $\a > \a_{2i+2}$, strictly semistable for $\a = \a_{2i+2}$ and unstable for $\a < \a_{2i+2}$.

Retain $C$, $\rho$ and analyze $B_\rho([C]_m)$. We have:

\begin{enumerate}
\item Let $C' = D \cup_{p'_1,p'_2} B'$ be a bicanonical genus $i$ bridge, $i\ge 2$, where $B'$ is a hyperelliptic
curve of genus $i$ and $p'_1, p'_2$ are interchanged by the hyperelliptic involution of $B'$.
Then
\[
\mu([C']_m, \rho) = \mu([C]_m, \rho)
\]
 In particular $C'$ is $\a$-Hilbert stable with respect to $\a
> \a_{2i+2}$, strictly semistable for $\a = \a_{2i+2}$ and unstable for $\a < \a_{2i+2}$.  Such a subcurve $B'$ (and sometimes
$C'$ itself by abusing terminology) is called a \emph{hyperelliptic genus $i$ bridge}.

\item Let $C''$ be a bicanonical genus $g$ curve obtained from $C'$ by replacing $B'$ by
an $A_{2i+1}$ singularity. That is, $C''$ is of genus $g$, has $A_{2i+1}$ singularity at $p''$ and admits a
partial normalization $\nu : (D, p'_1,p'_2) \to (C'', p'')$ mapping $p'_i$ to $p''$. Then
\[
\mu([C'']_m, \rho) = -\mu([C']_m, \rho)
\]
 In particular $C''$ is $\a$-Hilbert unstable with respect to $\a
> \a_{2i+2}$, strictly semistable for $\a = \a_{2i+2}$ and stable otherwise.
\end{enumerate}

\end{proposition}

%A {\it Weierstrass genus $i$ tail} consists of a genus $g-i$ curve $D$  meeting a hyperelliptic curve $T$ of genus $i$ in a node $p$ such that $p$ is a Weierstrass point of $T$.
% A {\it hyperelliptic genus $i$ bridge} consists of a genus $g-i-1$ curve $D$ meeting a hyperelliptic curve $B$ of genus $i$ in two nodes $p$ and $q$ that are permuted by the hyperelliptic involution of $B$.
%Preliminary analysis of basin of attraction of these curves suggests that a bicanonical Weierstrass genus $i$ tail (resp. hyperelliptic genus $i$ bridge)  becomes Hilbert strictly semistable and then unstable as $\a$ hits and goes below $\a_{2i+1}$  (resp. $\a_{2i+2}$): A sneak peak  of how the log minimal model program for $\Hg$ and $\Mg$ would progress, telling us when the log canonical model changes and which curves to discard and to acquire at the critical steps.

\begin{ack} The author thanks Brendan Hassett for sharing his ideas on the log minimal model program for the moduli of stable curves. Especially, the results in \S\ref{S:tails} were developed from his stability computation of genus two tails. The author thanks Sean Keel for pointing out inaccuracies in a preliminary version, and Jarod Alper and Maksym Fedorchuk for patiently explaining their heuristics of predicting critical values and change in singularity loci. He also greatly benefited from useful conversations with Valery Alexeev, Angela Gibney, Young-Hoon Kiem, Yongnam Lee, Ian Morrison, Jihun Park, David Smyth and Dave Swinarski. 
\end{ack}

\section{Intersection with vital curves}\label{S:intersection}
We shall use the isomorphism $\Hg \simeq \widetilde M_{0,2g+2} := \M_{0,2g+2}/\mfk S_{2g+2}$ throughout the paper, which grants us access to the  intersection theory on $\widetilde M_{0,n}$ worked out comprehensively in \cite{KMc}. Recall that a {\it vital curve} is an irreducible component of the locus in $\widetilde M_{0,2g+2}$ consisting of pointed curves with $\ge 2g-2$ nodes,
and any effective curve is conjectured to be an effective sum of vital curves ({\it Fulton Conjecture}, \cite{GKM, KMc, Gib}) which has been confirmed positively for $g \le 11$. A vital curve may be determined by $\{a, b, c, d\} \subset \ZZ_+$
such that $a+b+c+d = 2g+2$ \cite[4.1]{KMc}, and its intersection with a divisor $\sum_{k=2}^{g+1} r_k \widetilde B_k$ is given by
\begin{equation}\label{E:cdotb}
r_{a+b} + r_{b+c} + r_{a+c} - r_a - r_b - r_c - r_d \tag{$\dag$}
\end{equation}
\cite[1.3,4.4]{KMc}.

In this section, we shall derive the conjectural critical value formula (\ref{E:critical-values}) and the log discrepancy formula (\ref{E:Lai}) by applying Hassett-Keel principle.
 Recall that $L_\a^{[j]}$ denotes the pullback under $f_{g, \le j}= f_{g,j}\circ f_{g,j-1} \circ \cdots \circ f_{g,3}$
of the cycle theoretic image of $L_\a$ on $\Hg(\a_j)$.
Let $r_k^{[j]}$ denote the $\widetilde B_k$ coefficients of $L_\a^{[j]}$ i.e.
\[
L_\a^{[j]} = \sum_{k=2}^{g+1} r_k^{[j]} \widetilde B_k, \quad L_\a^{[2]} := L_\a.
\]
For notational convenience we also let $r_k$ denote $r_k^{[2]}$.
To prove Proposition~\ref{P:main1}, we shall use induction on $j$. Note that the statement is trivially true for $j = 2$. Assume that it is true for $j-1$. Hassett-Keel principle asserts that there exists a birational contraction $f_{g,j}$ that contracts the divisor $\widetilde B_j$ (image of it in $\Hg(\a_{j-1})$, to be more precise), so the extremal ray for $f_{g,j}$ is the class of the vital curve $\{a, b, j-a-b, 2g+2-j\}$. Let $C_{a,b,c}$ denote the vital curve associated to $\{a,b,c,2g+2-a-b-c\}$. Intersecting $C_{a,b,j-a-b}$ with $L_{\a}^{[j-1]}$, we obtain
\[
\begin{array}{c}
r_{a+b} + r_{j-a} + r_{j-b} - r_a - r_b - r_{j-a-b} - r_j \\
=
\left(\binom{a+b}2+\binom{j-a}2+\binom{j-b}2-\binom a2 - \binom b2 - \binom{j-a-b}2 \right) r_2 - r_j \\
\end{array}
\]
which after a simple combinatorics yields
\[
\binom j2 r_2 - r_j.
\]
Since the intersection of  $\widetilde B_j$ with $\{a,b,j-a-b,j\}$ is $-1$, the discrepancy formula is
\[
L_\a^{[j]} = L_\a^{[j-1]} + \left(\binom j2 r_2 - r_j\right) \widetilde B_j.
\]
Note that change occurs only in the coefficient of $\widetilde B_j$:
\[
r_j^{[j]} = \left(\binom j2 r_2 - r_j\right) + r_j = \binom j2 r_2.
\]
Hence the coefficient of the exceptional $\widetilde B_j$ is given by
\[
\begin{array}{lll}
c_j &=& r_j^{[j]} - r_{j}^{[j-1]} = \binom j2 r_2 - r_{j} \\
&=& \begin{cases}
\binom j2 \left(\frac{13}{4g+2}g + 2(\a-2)\right) - \frac{13}{4g+2}\frac{j-1}2\left( g - \frac{j-1}2\right) - \frac{\a-2}2, \quad \mbox{$j$ is odd} \\
\binom j2 \left(\frac{13}{4g+2}g + 2(\a-2)\right) - \frac{13}{4g+2}\frac{j}2\left( g +1 - \frac{j}2\right) - 2(\a-2), \quad \mbox{$j$ even}
\end{cases} \\
 &=& \begin{cases}
 \left( j(j-1) - \frac12\right) \a - \frac{3j^2 + 10 j - 21}{8}, \quad \mbox{$j$ odd}\\
 (j-2)\left[(j+1)\a - \left(\frac{3j}8+2\right)\right], \quad \mbox{$j$ even}
 \end{cases}
\end{array}
\]
Setting $c_j = 0$ and solving it for $\a$, we obtain the formula for critical values
(\ref{E:critical-values}) of Proposition~\ref{P:main1}.

 Now we are ready to prove Proposition~\ref{P:main1}.\,(3). Since $f_g: \Hg \to \Bg$  and $f_{g, \le g+1}: \Hg \dashrightarrow \Hg(\a_{g+1})$ both contract the boundary divisors $\widetilde B_3, \widetilde B_4, \dots, \widetilde B_{g+1}$ and no other divisors, $\Hg(\a_{g+1})$ and $\Bg$ are isomorphic away from a codimension two locus. We shall prove that an ample divisor on $\Hg(\a_{g+1})$ is pulled back to an ample divisor on $\Bg$. By the log discrepancy formula, the distinguished ample divisor on $\Hg(\a_{g+1})$ pulls back to $r_2 \sum_{k=2}^{g+1} \binom k2 \widetilde B_k$.

On the other hand, recall the geometric line bundles $L_{\vec{\bf x}}$ on $\bar M_{0,n}$ pulled back from the GIT quotients $(\bP^1)^{n}/\!\!/_{\vec{\bf x}}SL_2$ where ${\bf x} = (a_1, \dots, a_{n})$ denotes the linearization $\cO(a_1, \dots, a_n)$  used in the GIT construction \cite{AS}. Let $\varpi$ denote the projection
\[
(\bP^1)^n \to (\bP^1)^n/\mfk S_n \simeq \bP^n
\]
and note that $\varpi^*(\cO_{\bP^n}(1)) \simeq \cO(1,1,\dots,1)$. The natural ample line bundle on the GIT quotient $\Bg$ comes from $\cO_{\bP^{2g+2}}(1)$. Hence its pullback to $\bar M_{0,2g+2}$ is a positive rational multiple of the symmetric linearization $L_{(\frac1{g+1}, \dots, \frac1{g+1})}^{\otimes 2g+2}$, which is equal to
\begin{equation}\label{E:lin-bg}
\sum_{k=2}^{g+1} \frac{2k(k-1)}{2g+1} B_k.
\end{equation}
where $B_k = \sum_{|S| = k} D_{S,S^c}$ and $D_{S,S^c}$ is the divisor whose general pointed curve has labels of $S$ in one component and labels of $S^c$ on the other.
Kiem and Moon also made a similar computation \cite[Lemma~5.3]{KiemM}.
 The assertion now follows since (\ref{E:lin-bg}) is a symmetric divisor that descends to a multiple of $\sum_{k=2}^{g+1} \binom k2 \widetilde B_k$ on $\widetilde M_{0,2g+2}$.

\section{Modularity principle of Hassett and Keel}\label{S:GIT}
Hassett-Keel principle for $\Hg$ asserts that the intermediate log canonical models $\Hg(\a)$  are moduli spaces themselves parametrizing curves with prescribed singularities. These singularities are determined by the divisors contracted by $\Hg \dashrightarrow \Hg(\a)$. For instance, consider a general curve in $\widetilde B_{2s+1}$ which is a Weierstrass genus $s$ tail. Recall that it is a genus $g$ curve $D\cup_p T$ consisting of a genus $g-s$ curve $D$ and genus $s$ hyperelliptic curve $T$ meeting in one node such that the point $p$ of attachment is a Weierstrass point of $T$. When $\widetilde B_{2s+1}$ is contracted, such a curve $D \cup_p T$ must be replaced in the resulting moduli space $\Hg(\a_{2s+1})$ by a genus $g$ curve that depends only on $(D, p)$, and a natural candidate for the replacement is the curve $C'$ obtained from $D\cup_p T$ by {\it replacing} $T$ with an $A_{2s}$ singularity $y^2 = x^{2s+1}$ at $p$.\footnote{ This is a natural generalization of the picture  we described in \cite{HH1, HL3} in which the divisorial contraction $\Mg \to \Mg(9/11)$ (resp. $\Hg \to \Hg(9/11)$) replaces an elliptic tail with an ordinary cusp.} To explain how $\a$ appears as a parameter in the variation of GIT quotients of the Hilbert scheme, we recall the natural linearization of the Hilbert scheme. Let $Hilb_{g,\nu}$ denote the closure of the locus of $\nu$-canonical smooth genus $g$ curves in the Hilbert scheme $Hilb^{P}(\bP^N)$ parametrizing subschemes of $\bP(V) := \bP^N$ of Hilbert polynomial $P(m) = (2\nu m-1)(g-1)$. For $m \gg 0$, $Hilb_{g,\nu}$ admits embeddings
\[
\phi_m: Hilb_{g,\nu} \inj Gr(P(m), S^m V^*) \inj \bP(\bigwedge^{P(m)} S^m V^*)
\]
such that $\phi_m^*(\cO(+1))$ is a positive rational multiple of
\begin{equation}\label{E:lin}
(m-1)((6m\nu^2-2m\nu-2\nu+1)\lambda -
                \frac{m\nu^2}{2}\delta)
\end{equation}
where $\lambda$ is the determinant of the Hodge bundle and $\delta$ is the divisor of the singular curves \cite[P28. Equation~(5.3)]{HH2}. For a $\nu$-canonical curve $C \subset \bP^N$, we let $[C]_m$ denote the image in $ \bP(\bigwedge^{P(m)} S^m V^*)$ of $[C]\in Hilb_{g,\nu}$ under $\phi_m$ and call it the $m$th Hilbert point of $C$.

\begin{definition} $C$ is said to be $m$-Hilbert stable (resp. semistable, unstable) if $[C]_m$ is GIT stable (resp. semistable, unstable) with respect to the natural $SL(V)$ action on $ \bP(\bigwedge^{P(m)} S^m V^*)$.
\end{definition}

On the other hand, $K_{\overline{\mathcal M}_g} + \a \delta$ pulls back by the quotient map $Hilb_{g,\nu} \dashrightarrow \Mg$ to
$
13 \lambda - (2 - \a) \delta
$
which is proportional to (\ref{E:lin}) when $\nu = 2$ and
\begin{equation}\label{E:anm}\tag{\dag\dag}
m = \frac{3(2-\a)}{2(7-10\a)}, \quad \a = \frac{14m-6}{20m-3}.
\end{equation}
For meaningful values of $m$ and $\a$, this is a one-to-one order preserving correspondence between them. We say that a $\nu$-canonical curve $C \subset \bP^N$ is {\it $\a$-Hilbert stable (resp. semistable, unstable)} if it is $m$-Hilbert stable (resp. semistable, unstable) for the corresponding $m$.\footnote{This of course leaves  room for confusion. Rule of thumb is, $m$ is usually  an integer, such as `$6$-Hilbert stable' means $[C]_6$ is stable, and `$2/3$-Hilbert stable' means again $[C]_6$ is stable since $m(2/3) = 6$.}

\subsection{Cuspidal Tails}\label{S:tails}
We first consider the curves in $\widetilde B_{j}$, $j = 2b+1, b \ge 2$.
Let $R$ be a rational curve of genus $b$ with a single cusp $q$ whose local analytic equation is $y^2 = x^{j}$. Let $C = D\cup_p R$ be a bicanonical curve of genus $g$ consisting of $R$ and a genus $g-b$ curve $D$ meeting in a single node $p$. Restricting $\cO_C(1)$ to $R$ (resp. $D$), we find that it is of degree $4b-2$ (resp. $4g-4b-2$) and contained in a linear subspace of dimension $3b-1$ (resp. $3g-3b-1$). We can and shall choose coordinates such that
$R \subset \{x_{3b-1} = x_{3b} = \cdots = x_{3g-4} = 0\}$ and $D \subset \{x_0 = x_1 = \cdots = x_{3b-3} = 0\}$. $R$ may be parametrized by mapping $[s,t]$ to
\[
[s^{4b-2}, s^{4b-4}t^2, s^{4b-6}t^4, \cdots, s^{2b-2}t^{2b}, s^{2b-3}t^{2b+1}, s^{2b-4}t^{2b+2}, \cdots, t^{4b-2}]
\]
so that $R$ has a single cusp $ x_{b+1}^{2} = x_1^{2b+1}$ at $ q = [1, 0, \dots, 0]$, where we abused notation and let $x_1$ and $x_{b+1}$ denote their images in the completed local ring at $q$. Let $\rho$ denote the one-parameter subgroup with weights $$(0,2,4,6,\dots,2b, 2b+1, 2b+2, \dots, 4b-2, 4b-2, \dots, 4b-2).$$
The sum of these weights is $r := \binom{4b-1}2 - b^2 + (4b-2)(3g-3b-2)$.
We shall fix the $\rho$-weighted GLex order on the monomials.
\begin{lemma} The sum $w_{R,\rho}(m)$ of the weights of the degree $m$ monomials of  in $x_0, \dots, x_{3b-2}$ that are not in the initial ideal of $R$ is
\[
w_{R,\rho}(m) = (8b^2-8b+2)m^2 + (2b-1)m - b^2.
\]
\end{lemma}

\begin{proof} In general, weight computation of this sort can be accomplished by using Gr\"obner basis, but in this case there is a more elementary solution since $R$ admits a parametrization. Let $P_R(m) = (4m-1)(b-1) + 2m$, the Hilbert polynomial of $R$.  A monomial of degree $m$ pulls back to one of the following $m(4b-2)+1-b$ monomials
\[
s^{m(4b-2)-i}t^i, \quad i = 0, 2, 4, \dots, 2b, 2b+1, 2b+2, \dots, m(4b-2).
\]
If $\prod_{i\in I,|I|=m} x_i$ and $\prod_{i \in J,|J|=m} x_i$ pull back to the same monomial, then $\prod_{i\in I,|I|=m} x_i - \prod_{i \in J,|J|=m} x_i$ is in the initial ideal $in_\rho(R)$ of the ideal of $R$ with respect to the $\rho$-weighted GLex order. It follows that each $s^{m(4b-2)-i}t^i$ appears at most once
among the pullbacks of degree $m$ monomials not in the initial ideal $in_\rho(R)$. Since $m(4b-2)+1-b$ equals $P_R(m)$, it has to appear in the set exactly once. Therefore,
\[
\begin{array}{clll}
w_{R,\rho}(m) & = & \sum_{k=0}^b 2k + \sum_{k=2b+1}^{m(4b-2)}k \\
& = & (8b^2-8b+2)m^2 + (2b-1)m - b^2.
\end{array}
\]
\end{proof}
On the other hand, the contribution from $D$ to the total weight is
\[
\begin{array}{clllll}
w_{D,\rho}(m) & = & (4b-2)m \cdot h^0((\cO_C^{\otimes 2}|D)^{\otimes m} (-p))\\
& = &(4b-2)m((4m-1)(g-b-1)+2m-1)
\end{array}
\]
since $\rho$ acts on $D$ trivially with constant weight $4b-2$. Combining these weights and the average weight, we obtain the Hilbert-Mumford index of $C$: \small
\[
\begin{array}{lllllllllll}
\mu([C]_m, \rho) & = & \frac{mP(m)}{N+1}r - w_{R,\rho}(m) - w_{D,\rho}(m)
 \\
& = & \frac13m(4m-1)(\frac12(4b-1)(4b-2) - b^2 + (4b-2)(3g-3b-2)) \\
&& - ((8b^2-8b+2)m^2 + (2b-1)m - b^2) \\
&& - ((4b-2)m((4m-1)(g-b-1)+2m-1))\\
& = & \frac13\left((4b^2-8b+2)m^2 + (-7b^2+8b-2)m + 3b^2\right) \\
& = & \frac13(m-1) ((4b^2-8b+2)m - 3b^2)
\end{array}
\]\normalsize
Using the relation (\ref{E:anm}) from \S\ref{S:GIT} and substituting $j = 2b+1$, we can rewrite it as
\[
\begin{array}{cllll}
\mu([C]_m, \rho) & = &
 \frac{17\a-8}{8(7-10\a)^2}\left((8j^2 - 8j - 4)\a - (3j^2+ 10 j - 21)\right) \\
 & = & \frac{(8j^2 - 8j - 4)(17\a-8)}{8(7-10\a)^2}(\a - \a_j).
 \end{array}
 \]
This completes the proof of Proposition~\ref{P:tails}.

\subsubsection{Basin of attraction of cuspidal tails}\label{S:boa-ctails}
The local versal deformation space of $q$ is given by
\[
x_{2b+1}^2 = x_1^{2b+1} + c_{2b-1} x_1^{2b-1} + c_{2b-2} x_1^{2b-2} + \cdots +c_0
\]
Since $\rho$ acts on $x_1$ and $x_{2b+1}$ with weights $2$ and $2b+1$ respectively, it acts on $c_i$ with positive weight $4b+2-2i$, $0 \le i \le 2b-1$. Hence the basin of attraction contains arbitrary smoothing of $q$. By considering the local stable reduction  \cite[\S~6.2.2]{Has00}, we can deduce that if $D\cup_p T$ is in the basin $B_{\rho}([C]_m)$, then $T$ must be hyperelliptic and $p$ is a Weierstrass point of $T$. Indeeed, consider the isotrivial family $\mathcal C \to B=\spec k[[t]]$ whose general member is $\rho(t). (D\cup_p T)$ and the special member is $C$. Stable reduction of $\mathcal C$ yields $\mathcal C' \to B'$, $B' \to B$ a finite covering, whose general member is isomorphic to that of $\mathcal C$ and the special member is $D\cup_{p'} T'$ where $T'$ is  hyperelliptic and $p'$ is a Weierstrass point of $T'$: This is precisely the content of\cite[\S~6.2.2]{Has00}. By the separability of $\overline{\mathcal M}_g$, it follows that  $T$ is hyperelliptic and $p$ is a Weierstrass point.

On the other hand, $\rho$ acts trivially on the local versal deformation of the node $p$ and the basin of attraction does not contain any smoothing of $p$.

 \begin{figure}[ht]
\centerline{\scalebox{0.6}{\psfig{figure=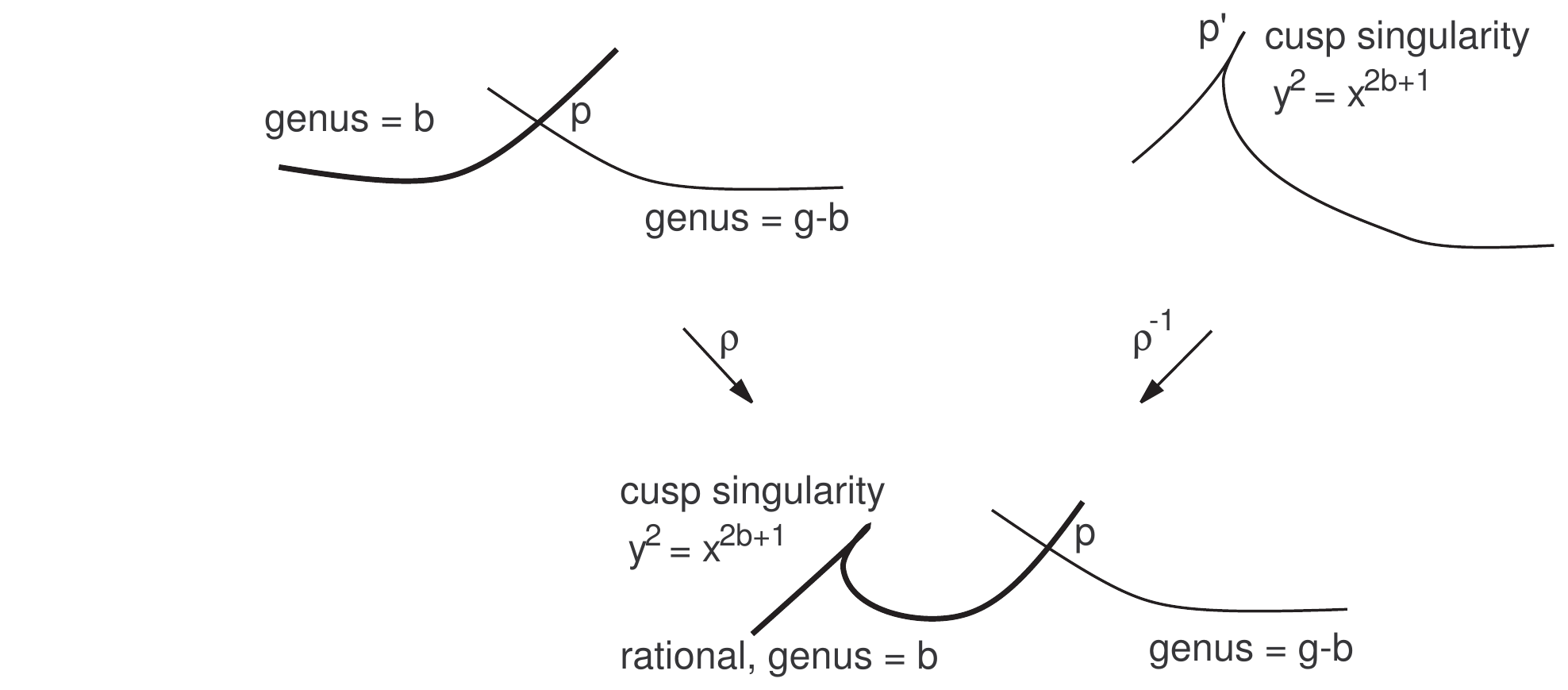}}}
\caption{Basin of attraction of cuspidal tail}
\label{F:cusps}
\end{figure}

%\capdraw{F:cusps}{Basin of attraction of cuspidal tail}{.6}{cusps.pdf}
%{
%}% \begin{figure}[ht]
%\centerline{\scalebox{0.6}{\psfig{figure=cusps.pdf}}}
%\caption{Basin of attraction of cuspidal tail}
%\label{F:cusps}
%\end{figure}

\subsection{Nodal Bridges}\label{S:bridges}
In this section, we consider the case $\widetilde B_{j}$  for even $j = 2b+2, b \ge 2$.
Let $R = R_1\cup_{a_1} R_2$ be a curve of arithmetic genus $b$ consisting of two  rational curves meeting in an $A_{2b+1}$ singularity at $a_1$.
Let $C = R \cup_{a_0,a_2} D \subset \bP^{3g-4}$ be a bicanonical curve of genus $g$ consisting of $R$ and a genus $g-b-1$ curve $D$ meeting in two nodes $a_0$ and $a_2$.
By Riemann-Roch, we have
$h^0(R, \cO(1)|R) = 3b+1$ and $h^0(D, \cO(1)|D) = 3g-3-(3b-1)$ and may choose coordinates such that
\[
R \subset \{x_{3b+1} = x_{3b+2} = \cdots = x_{3g-3} = 0\}
\]
and
\[
D \subset \{x_0 = \dots = x_{3b-2} = 0\}.
\]
The rational components $R_i$ of $R$ are embedded as degree $2b$ curves
in $\bP^{N-3b} = \{x_{3b+1} = \cdots = x_N = 0\} \subset \bP^N$, and
can be parametrized as follows so that they meet in an $A_{2b+1}$ singularity: we map $(s_i, t_i)$ to
\small
%\[
%\begin{array}{lrrrrrrrrrrrrrrrrrrrrrrrrrr}
%\left[s_1, t_1\right]  \mapsto  \\
%(s_1^{2b}, & s_1^{2b-1}t_1, &\dots, &s_1^{b}t_1^{b}, & s_1^{b-1}t_1^{b+1}, & \dots, & s_1t_1^{2b-1},
%& 0,                 & \dots,  &   0,           & t_1^{2b}, & 0,      &0, \dots, 0 )\\
%\left[s_2, t_2\right] \mapsto  \\
%(s_2^{2b},& s_2^{2b-1}t_2, &\dots, &s_2^{b}t_2^{b},  &  0,   & \dots,
%& 0,   & s_2^{b-1}t_2^{b+1}, & \dots, & s_2 t_2^{2b-1}, & 0,        & t_2^{2b},& 0, \dots, 0) \\
%(x_0, &x_1,& \dots, &x_b, & x_{b+1}, & \dots, &x_{2b-1}, & x_{2b}, &\dots,
%& x_{3b-2}, &, x_{3b-1}, & x_{3b}, & x_{3b+1}, \dots, x_{N})\\
%\end{array}
%\]
\[
\begin{tabular}{|c|c|c|c|c|c|c|c|c|c|c|}
    \hline
    % after \\: \hline or \cline{col1-col2} \cline{col3-col4} ...
   $ x_0 $ & $ \dots $ & $ x_b $ & $ x_{b+1} $ & $ \dots $ & $ x_{2b-1} $ & $ x_{2b} $ & $ \dots $ & $ x_{3b-2} $ & $ x_{3b-1} $ & $ x_{3b}$
     \\
\hline
   $ s_1^{2b}  $ & $\dots $ & $s_1^{b}t_1^{b} $ & $ s_1^{b-1}t_1^{b+1} $ & $ \dots $ & $ s_1t_1^{2b-1}
$ & $ 0 $ & $ \dots  $ & $   0           $ & $ t_1^{2b} $ & $ 0 $\\
$ s_2^{2b} $ & $\dots $ & $s_2^{b}t_2^{b}  $ & $  0   $ & $ \dots
$ & $ 0   $ & $ s_2^{b-1}t_2^{b+1} $ & $ \dots $ & $ s_2 t_2^{2b-1} $ & $ 0        $ & $ t_2^{2b} $\\
    \hline
  \end{tabular}
\]
%\[
%\begin{array}{lrrrrrrrrrrrrrrrrrrrrrrrrrr}
%\left[s_1, t_1\right]  \mapsto  \\
%(s_1^{2b},  &\dots, &s_1^{b}t_1^{b}, & s_1^{b-1}t_1^{b+1}, & \dots, & s_1t_1^{2b-1},
%& 0, & \dots,  &   0,           & t_1^{2b}, & 0 )\\
%\left[s_2, t_2\right] \mapsto  \\
%(s_2^{2b}, &\dots, &s_2^{b}t_2^{b},  &  0,   & \dots,
%& 0,   & s_2^{b-1}t_2^{b+1}, & \dots, & s_2 t_2^{2b-1}, & 0,        & t_2^{2b}) \\
%(x_0, &x_1,& \dots, &x_b, & x_{b+1}, & \dots, &x_{2b-1}, & x_{2b}, &\dots,
%& x_{3b-2}, &, x_{3b-1}, & x_{3b}, & x_{3b+1}, \dots, x_{N})\\
%\end{array}
%\]
\normalsize
in $\{x_{3b+1} = \cdots = x_N = 0\}$.
In this coordinate system, we have
\[
\begin{array}{l}
a_0 = (\overbrace{0, \dots, 0}^{3b-2}, 1, 0, 0, \dots, 0) \\
a_2 = (0, \dots, 0, 0, 1, 0, \dots, 0) \\
a_1 = (1, 0, \dots, 0, 0, 0, \dots, 0).
\end{array}
\]
Let $\rho$ be the 1-PS with weights $$(0,1,2,3, \dots, 2b-1, b+1, b+2, \cdots, 2b-1, 2b, 2b \dots, 2b).$$
The sum of the weights of $\rho$ is
\[
\binom{2b+1}2 + \binom{b+1}2 + b^2 + 2b(3g-3b-4).
\]

\begin{lemma}
The total $\rho$-weight $w_{R,\rho}(m)$ of the degree $m$ monomials in $x_0, \dots, x_{3b}$ not in the initial ideal with respect to $\rho$-weighted GLex is
\[
w_{R,\rho}(m) = 4b^2 m^2  + 2b m - \binom{b+1}{2}
\]
\end{lemma}
\begin{proof} Let $\cal S$ denote the monomials of degree $m$ in $x_0, \dots, x_{3b}$ not in the initial ideal $in_\rho(R)$. For $k \le b$, there are at most one monomial in $\cal S$ of weight $k$, namely $x_0^{m-1}x_k$. In general, we claim that $\cal S$ contains at most two monomials of weight $k$. Suppose that $M$ and $M'$ are monomials of weight $k$. If $k \ge bm+1$, then a monomial of weight $k$ must involve $x_j$ for some $j \ge b+1$ and hence vanishes on $R_1$ or $R_2$. If $M$ and $M'$ both vanish on the same component, then $M-M'$ is in the ideal of $R$ and one of them is in the initial ideal. Hence $\cal S$ may contain at most one monomial of weight $k$  that vanishes on $R_1$ but not on $R_2$, and another with the opposite vanishing property. If $k \le bm$, there could be a monomial $M''$ of weight $k$ comprised of $x_0, \dots, x_b$ only. These do not vanish on either components. But in this case, $M + M' - M''$ is in the ideal of $R$ and one of them is in the initial ideal.
Counting these possible monomials in $\cal S$, there are $2bm-b$ monomials of weight $b+1 \le k \le 2bm$ that vanishes on $R_i$ for each $i$ plus the $b+1$ monomials $x_0^m, x_0^{m-1}x_1, \dots, x_0^{m-1}x_b$, totaling $2 (2bm-b) + b+1 = 4bm + 1-b$. But this equals $P_R(m)$, the Hilbert polynomial of $R$, and we conclude that these are precisely the monomials in $\cal S$. Summing up their weights, we obtain
\[
2\cdot \sum_{k=b+1}^{2bm}k + \sum_{k=0}^b k = 2 \cdot \sum_{k=1}^{2bm}k - \sum_{k=0}^bk = 2bm(2bm+1) - \frac{b(b+1)}2.
\]
\end{proof}

Since $\rho$ acts on $D$ trivially with weight $2b$, the contribution from $D$ is
\[
2bm \, h^0(\cO(m)|D (-a_0 - a_2)) = 2bm ((4m-1)(g-b-1) - 1).
\]
The (normalized) Hilbert-Mumford index of $C$ with respect to $\rho$ is
\[\displaystyle{
\begin{array}{lll}
\mu([C]_m, \rho) & = &\frac{mP(m)}{N+1} \left(\binom{2b+1}2 + \binom{b+1}2 + b^2 + 2b(3g-3b-4)\right) \\
& & - (4b^2 m^2 + 2b m - \binom{b+1}2) - 2bm((4m-1)(g-b-1)-1) \\
& = & \frac23b(b-1) m^2 - \frac16b(7b-1)m + \binom{b+1}2 \\
& = & \frac16(m-1)(4b(b-1)m-3b(b+1)).
\end{array}}
\]
Using the relation (\ref{E:anm}) and $j = 2b+2$, we obtain
\[
\begin{array}{cllll}
\mu([C]_m, \rho) & = &
 \frac{(j-2)(17\a-8)}{16(7-10\a)^2}\left(8(j+1)\a - (3j+16)\right) \\
 & = & \frac{(j^2-j-2)(17\a-8)}{2(7-10\a)^2}(\a - \a_j).
 \end{array}
 \]

\subsubsection{Basin of attraction of nodal bridges}\label{S:boa-bridges}
We apply similar arguments as in \S~\ref{S:boa-ctails} to analyze the $\rho$-action on the local versal deformation of the singularities, and conclude that the basin of attraction of $C$ with respect to $\rho$ contains arbitrary smoothing of $a_1$ but no smoothing of $a_0$ and $a_2$. If $D\cup_{a_0, a_2} E$ is in the basin of attraction of $B_\rho([C]_m)$, then by \cite[\S~6.2.1]{Has00}, $E$ is hyperellliptic of genus $b$ and $a_0, a_2$ are interchanged by the hyperelliptic involution.
On the other hand, the basin of attraction of $C$ with respect to $\rho^{-1}$ would contain the curves obtained from $C$ by replacing the bridge $R$ by a $(b+1)$-fold node. This completes the proof of Proposition~\ref{P:bridges}.

\begin{remark} Alper, Fedorchuk and Smyth  \cite{AFS} observes that $13\lambda - (2 - \alpha)\delta $ is positively proportional to $\frac{13\alpha-13}{2-\alpha}\lambda + \lambda_2$, and the Hilbert-Mumford index computation boils down to computing
$
\mu^{\lambda}([C],\rho)
$ and $\mu^{\lambda_2}([C],\rho)$. The upshot is that $\lambda_m | [C] = \bigwedge H^0(C, \omega_C^{\otimes m})$, and when $\rho$ is the 1-PS coming from $Aut(C)$, $
\mu^{\lambda_m}([C],\rho)$ is simply the $\rho$-weight of the $\bG_m$ action on 
 $\bigwedge H^0(C, \omega_C^{\otimes m})$. This enables them to compute the Hilbert-Mumford index without  working with explicit parametrization. They also formulate the modularity principle precisely and obtain a  comprehensive outline of the whole program, predicting when various singularities would appear. 
\end{remark}

%\capdraw{F:bridges}{Basin of attraction of nodal bridges}{.6}{bridges.pdf}
%{
%}%

\begin{figure}[ht]
\centerline{\scalebox{0.6}{\psfig{figure=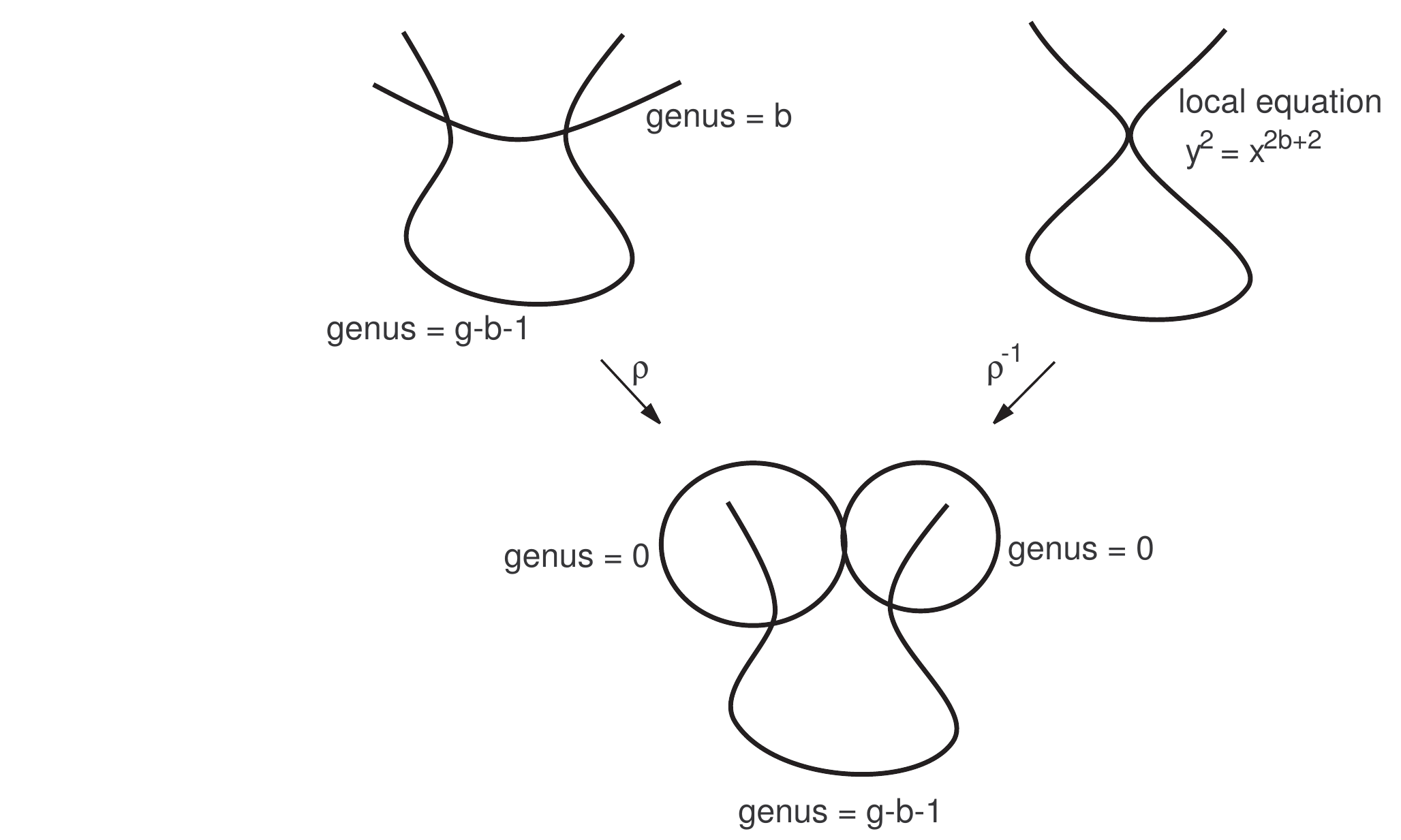}}}
\caption{Basin of attraction of nodal bridges}
\label{F:bridges}
\end{figure}

\bibliographystyle{alpha}
\bibliography{entwurf}

\def\cprime{$'$}
\begin{thebibliography}{GHvdP88}

\bibitem[AFS10]{AFS}
Jarod Alper, Maksym Fedorchuk, and David~Ishii Smyth.
\newblock Singularities with ${G}_m$-action and the log minimal model program
  for $\bar{M}_g$, 2010.
\newblock arXiv:1010.3751v1 [math.AG].

\bibitem[AL02]{AL}
D.~Avritzer and H.~Lange.
\newblock The moduli spaces of hyperelliptic curves and binary forms.
\newblock {\em Math. Z.}, 242(4):615--632, 2002.

\bibitem[AS]{AS}
Valery Alexeev and David Swinarski.
\newblock Nef divisors on $\bar {M}_{0,n}$ from {GIT}.
\newblock arXiv:0812.0778v1 [math.AG].

\bibitem[GHvdP88]{GvdP}
L.~Gerritzen, F.~Herrlich, and M.~van~der Put.
\newblock Stable {$n$}-pointed trees of projective lines.
\newblock {\em Nederl. Akad. Wetensch. Indag. Math.}, 50(2):131--163, 1988.

\bibitem[Gib09]{Gib}
Angela Gibney.
\newblock Numerical criteria for divisors on {$\overline M_g$} to be ample.
\newblock {\em Compos. Math.}, 145(5):1227--1248, 2009.

\bibitem[GKM02]{GKM}
Angela Gibney, Sean Keel, and Ian Morrison.
\newblock Towards the ample cone of {$\overline M\sb {g,n}$}.
\newblock {\em J. Amer. Math. Soc.}, 15(2):273--294 (electronic), 2002.

\bibitem[Has00]{Has00}
Brendan Hassett.
\newblock Local stable reduction of plane curve singularities.
\newblock {\em J. Reine Angew. Math.}, 520:169--194, 2000.

\bibitem[Has03]{Has1}
Brendan Hassett.
\newblock Moduli spaces of weighted pointed stable curves.
\newblock {\em Adv. Math.}, 173(2):316--352, 2003.

\bibitem[Has05]{Has}
B.~Hassett.
\newblock Classical and minimal models of the moduli space of curves of genus
  two.
\newblock In {\em Geometric methods in algebra and number theory}, volume 235
  of {\em Progress in mathematics}, pages 160--192. Birkh\"auser, Boston, 2005.

\bibitem[HH08]{HH2}
Brendan Hassett and Donghoon Hyeon.
\newblock Log minimal model program for the moduli space of curves: the first
  flip, 2008.
\newblock submitted, arXiv:0806.3444v1 [math.AG].

\bibitem[HH09]{HH1}
Brendan Hassett and Donghoon Hyeon.
\newblock Log canonical models for the moduli space of curves: the first
  divisorial contraction.
\newblock {\em Trans. Amer. Math. Soc.}, 361(8):4471--4489, 2009.

\bibitem[HL07]{HL1}
Donghoon Hyeon and Yongnam Lee.
\newblock Stability of tri-canonical curves of genus two.
\newblock {\em Math. Ann.}, 337(2):479--488, 2007.

\bibitem[HL10a]{HL4}
Donghoon Hyeon and Yongnam Lee.
\newblock Birational contraction of genus two tails in the moduli space of
  genus four curves {I}, 2010.
\newblock arXiv:1003.3973v1 [math.AG].

\bibitem[HL10b]{HL3}
Donghoon Hyeon and Yongnam Lee.
\newblock A new look at the moduli space of stable hyperelliptic curves.
\newblock {\em Math. Z.}, 264(2):317--326, 2010.

\bibitem[JA]{ASvdW}
Frederick van der~Wyck Jarod~Alper, David Ishii~Smyth.
\newblock Weakly proper moduli stacks of curves.
\newblock arXiv:1012.0538v2 [math.AG].

\bibitem[Kap93a]{Kap}
M.~M. Kapranov.
\newblock Chow quotients of {G}rassmannians. {I}.
\newblock In {\em I. {M}. {G}el\cprime fand {S}eminar}, volume~16 of {\em Adv.
  Soviet Math.}, pages 29--110. Amer. Math. Soc., Providence, RI, 1993.

\bibitem[Kap93b]{Kap1}
M.~M. Kapranov.
\newblock Veronese curves and {G}rothendieck-{K}nudsen moduli space {$\overline
  M_{0,n}$}.
\newblock {\em J. Algebraic Geom.}, 2(2):239--262, 1993.

\bibitem[KM96]{KMc}
Sean Keel and James McKernan.
\newblock Contractivle extremal rays on $\bar{M}_{0,n}$, 1996.
\newblock arXiv:alg-geom/9607009v1.

\bibitem[KM10]{KiemM}
Young-Hoon Kiem and Han-Bom Moon.
\newblock Moduli spaces of weighted pointed stable rational curves via {GIT},
  2010.
\newblock arXiv:1002.2461v2 [math.AG].

\bibitem[MF]{FS}
David Ishii~Smyth Maksym~Fedorchuk.
\newblock Alternate compactifications of moduli spaces of curves.
\newblock arXiv:1012.0329v1 [math.AG].

\end{thebibliography}

\end{document}